\documentclass[10pt,a4paper]{amsart}
\usepackage[utf8]{inputenc}
\usepackage[T1]{fontenc}
\usepackage{amsmath}
\usepackage{amssymb}
\usepackage{graphicx}
\usepackage{hyperref}
\usepackage{cleveref}
\usepackage{tikz-cd}
\usepackage[margin=3.6cm]{geometry}

\newtheorem{thm}{Theorem}
\newtheorem{cor}[thm]{Corollary}
\newtheorem{lem}[thm]{Lemma}
\newtheorem{prop}[thm]{Proposition}
\theoremstyle{definition}
\newtheorem{rem}[thm]{Remark}

\title{Universal covers of non-negatively curved manifolds and formality}

\begin{document}
 \author{Aleksandar Milivojevi\'c}
 \address{University of Waterloo, Faculty of Mathematics}
 \email{amilivoj@uwaterloo.ca}
 
 \begin{abstract} We show that if the universal cover of a closed smooth manifold admitting a metric with non-negative Ricci curvature is formal, then the manifold itself is formal. We reprove a result of Fiorenza--Kawai--L\^{e}--Schwachh\"ofer, that closed orientable manifolds with a non-negative Ricci curvature metric and sufficiently large first Betti number are formal. Our method allows us to remove the orientability hypothesis; we further address some cases of non-closed manifolds.\end{abstract}

\maketitle


A long-standing theme in the study of smooth manifolds is to find topological obstructions to the existence of certain geometric structures. A cornerstone result is that of Deligne--Griffiths--Morgan--Sullivan \cite{DGMS75}, that  manifolds admitting a complex structure satisfying the $\partial \overline{\partial}$-lemma are formal. That is, the (weak) homotopy type of its commutative differential graded algebra of de Rham forms is that of its de Rham cohomology algebra equipped with trivial differential. 

Certain such obstructions behave well under non-zero degree maps of closed orientable manifolds. For example, as shown by Taylor, non-trivial triple Massey products pull back non-trivially under non-zero degree maps; Crowley--N\"ordstrom's Bianchi--Massey tensor \cite{CN20} exhibits the same behavior, as does non-formality itself \cite{MSZ23}.

In this note, we will observe that using the Cheeger--Gromoll splitting theorem, one can employ formality as an obstruction to non-negative curvature which is preserved under potentially infinite-degree maps:

\begin{thm}\label{universal} A connected closed manifold admitting a metric with non-negative Ricci curvature is formal if its universal covering space is formal. \end{thm}


Before proving this, we record the relevant parts of the statements of the Gromoll--Cheeger splitting theorem, and the formal domination theorem of \cite{MSZ23}, that we will use:

\begin{thm}(\cite{CG71}, \cite{CG72}, see formulation in \cite[Corollary 6.67, (b) and (c)]{Be07})
Let $(M,g)$ be a closed connected Riemannian manifold with non-negative Ricci curvature. Then a finite covering of $M$ is diffeomorphic to $\hat{M} \times T^q$, where $\hat{M}$ is a closed simply connected manifold, and $T^q$ is a torus of some dimension $q$. \end{thm}

\begin{thm}\label{formaldomination}(\cite[Theorem A]{MSZ23}) Suppose $Y \xrightarrow{f} X$ is a proper smooth map between smooth orientable manifolds which contains some (hence any) rational Borel--Moore fundamental class of $X$ in the image of the induced map on rational Borel--Moore homology. Then, if $Y$ is formal, $X$ is also formal. In particular, if $Y$ and $X$ are closed orientable manifolds and $f$ is a non-zero degree map, then formality of $Y$ implies formality of $X$. \end{thm}

As a consequence of \Cref{formaldomination}, we have the following useful lemma:

\begin{lem}\label{doublecoverlemma} If the orientable double cover $M'$ of a smooth non-orientable (not necessarily compact) manifold $M$ is formal, then $M$ is formal as well. \end{lem}

\begin{proof} Consider the total space of the non-orientable real line bundle over $M$ with first Stiefel--Whitney class $w_1$ equal to $w_1(TM)$ (i.e., the orientation line bundle). This is an orientable manifold which is doubly covered by the pullback of the line bundle over the orientable double cover $M'$ of $M$; namely, this is the total space of the trivial real line bundle over $M'$. The covering map is proper, and it is surjective on top-degree rational Borel--Moore homology, as the Borel--Moore fundamental class of the domain maps (up to sign) to twice the Borel--Moore fundamental class of the target. Hence the total space of the orientation line bundle over $M$ is formal if the total space of the trivial line bundle over $M'$ is formal. The formality of the total space of a vector bundle is equivalent to the formality of its base space, since the inclusion of the zero section is a homotopy equivalence. \end{proof}

\noindent \emph{Proof of \Cref{universal}}. Indeed, there is some finite cover of our manifold $M$ which is of the form $T^q \times \hat{M}$ for a simply connected $\hat{M}$. Since $T^q$ is formal and a product of two spaces is formal if and only if each factor is formal, the formality of this cover is equivalent to the formality of $\hat{M}$, which is equivalent to the formality of the universal cover $\mathbb{R}^q \times \hat{M}$. Hence, if the universal cover is formal, $T^q \times \hat{M}$ is formal. Then by \Cref{formaldomination}, if $M$ is orientable, it is formal. If $M$ is non-orientable, we consider its orientable double cover $M'$ endowed with the pullback metric. The universal cover of $M'$ coincides with that of $M$. Therefore, by the above, we conclude that if the universal cover of $M$ is formal, then $M'$ is formal, and hence $M$ is formal by \Cref{doublecoverlemma}. \qed

\begin{rem} As a consequence, we recover the well-known fact that non-toral nilmanifolds do not admit metrics with non-negative Ricci curvature, as they are not formal. More generally, any non-formal aspherical closed Riemannian manifold has a direction of negative Ricci curvature. \end{rem}

We reobtain the result that closed orientable Riemannian manifolds with non-negative Ricci curvature and sufficiently large first Betti number are formal, originally due to Fiorenza--Kawai--L\^{e}--Schwachh\"ofer \cite[Corollary 5.4]{FKLS21}. They prove the result more generally for closed orientable manifolds all of whose harmonic 1-forms are parallel, using their theory of Poincar\'e differential algebras of Hodge type and a generalization of the Cheeger--Gromoll splitting theorem proved therein. In the case of non-negative Ricci curvature, we can remove the orientability hypothesis. Below we will also consider cases in which one can remove the closedness assumption on the manifold.

\begin{prop}(cf. \cite[Corollary 5.4]{FKLS21})\label{largeb1closed} \ A connected closed (not necessarily orientable) manifold $M$ of dimension $n$ with $b_1 \geq n-6$, which admits a metric of non-negative Ricci curvature, is formal. \end{prop}

\begin{proof} By the Cheeger--Gromoll splitting theorem, the universal cover of $M$ is $\hat{M} \times \mathbb{R}^q$, where $q \geq b_1(M)$ and $\hat{M}$ is a closed simply connected manifold. Since $b_1(M) \geq n-6$, the dimension of $\hat{M}$ is at most six. Hence $\hat{M}$ is formal \cite{M79}, and so $\hat{M} \times \mathbb{R}^q$ is formal. Now by \Cref{universal}, we conclude that $M$ is formal. \end{proof}

In this generality, \Cref{largeb1closed} is sharp, as there exists a simply connected closed non-formal seven--manifold with positive Ricci curvature (in fact, admitting a positive Einstein metric), see \cite[manifold $Q(1,1,1)$ on p.2, p.5]{FFKM23}. Cases of \Cref{largeb1closed} with larger bounds on $b_1$ had been known earlier. Indeed, as was known to Bochner, a closed Riemannian manifold with non-negative Ricci curvature has $b_1 \leq n$, and if $b_1 = n$, it is isometric to a flat torus, which is formal.  Furthermore, Kotschick proved that a closed orientable $n$--manifold with a non-negative Ricci curvature metric and $b_1 = n-2$ is in fact geometrically formal \cite[Proposition 16]{Ko17}, i.e. it admits a metric wherein the harmonic forms are closed under wedge product, implying formality. As remarked in loc. cit., if $b_1 \geq n-1$ on a closed orientable $n$-manifold with a non-negative Ricci curvature metric, then in fact $b_1 = n$. 

\vspace{1em}

It is somewhat surprising that in the presence of non-negative Ricci curvature, a large first Betti number ensures formality, since often in other scenarios one hopes to argue formality using instead a scarcity of cohomology (as in e.g. \cite{M79}). 

\subsection*{Compact manifolds with boundary}

For a compact manifold $M$ with boundary $\partial M$, denote by $D(M)$ its double, i.e. the closed manifold obtained by gluing two copies of $M$ by the identity along $\partial M$. If $M$ is oriented, then putting the opposite orientation on the second copy of $M$ induces an orientation on $D(M)$. We partially extend \Cref{largeb1closed} to compact manifolds with boundary. We made use of the observation that the natural ramified double cover $D(M) \to M$ is a retraction for the inclusion $M \hookrightarrow D(M)$.

\begin{lem}\label{firstbetti} We have $b_1(D(M)) \geq b_1(M)$. Furthermore, if $D(M)$ is formal, then $M$ is formal. \end{lem}

\begin{proof} The homology of $D(M)$ surjects onto that of $M$, as $D(M)$ retracts onto $M$. Furthermore, the retract of a formal space is formal (see \cite[Example 2.88]{FOT08}, or apply \cite[Theorem B]{MSZ23}).  \end{proof}



\begin{cor}\label{compactwithboundary} Let $M$ be a connected compact $n$--manifold with boundary $\partial M$, with $b_1(M) \geq n-6$. If $M$ admits a Riemannian metric with positive Ricci curvature in which $\partial M$ is strictly convex (i.e. the second fundamental form along the boundary is positive definite), then $M$ is formal. \end{cor}

\begin{proof} Strict convexity of the boundary ensures that the conditions of Perelman's gluing construction \cite[Section 4]{Pe97}  are satisfied by two copies of $M$ (where we identify the boundaries via the identity map); we thus have that $D(M)$ admits a metric of positive Ricci curvature. The double $D(M)$ is a closed $n$--manifold, with $b_1(D(M)) \geq b_1(M) \geq n-6$ by the first part of \Cref{firstbetti}, so it is formal by \Cref{largeb1closed}. Then by the second part of \Cref{firstbetti}, $M$ is also formal. \end{proof}


In fact, $b_1(D(M))$ is determined by $b_1(M)$ and the dimension of the relative rational homology $H_1(M, \partial M)$, giving us the following generalization:

\begin{cor}\label{cor7} Let $M$ be a connected compact $n$--manifold with boundary $\partial M$, with $b_1(M) \geq n-6-\dim H_1(M, \partial M)$. If $M$ admits a Riemannian metric with positive Ricci curvature in which $\partial M$ is strictly convex, then $M$ is formal. \end{cor}

\begin{proof} Denote by $i$ the inclusion $\partial M \hookrightarrow M$. For simplicity, we denote the induced map on homology by $i$ as well. From the Mayer--Vietoris long exact sequence in rational homology $$ \cdots \to H_*(\partial M) \to H_*(M) \oplus H_*(M) \to H_*(D(M)) \to H_{*-1}(\partial M) \to \cdots $$ we obtain the exact sequence $$0 \to \ker i_{*} \to H_*(\partial M) \to H_*(M) \oplus H_*(M) \to H_*(D(M)) \to \ker i_{*-1} \to 0,$$ giving us $$2b_1(M) = b_1(\partial M) - \dim \ker i_{1} - \dim \ker i_{0} + b_1(D(M)),$$ where $\ker i_{0}$ and $\ker i_{1}$ denote the kernels of the maps induced by $i$ in degrees 0 and 1, respectively. On the other hand, from the long exact sequence in rational homology for the pair $(M, \partial M)$ we obtain the exact sequence $$0 \to \ker i_{*} \to H_*(\partial M) \to H_*(M) \to H_*(M, \partial M) \to \ker i_{*-1} \to 0,$$ giving us $$b_1(M) = b_1(\partial M) - \dim \ker i_{1} - \dim \ker i_{0} + \dim H_1(M, \partial M).$$ 

Combining these two equalities we obtain $b_1(D(M)) - b_1(M) = \dim H_1(M, \partial M)$, and we apply \Cref{compactwithboundary}. \end{proof}

\subsection*{Complete open manifolds} We also have the following analogue of \Cref{universal}:

\begin{thm}\label{universalsectional} A connected manifold, without boundary, admitting a complete metric with non-negative sectional curvature is formal if its universal covering space is formal. \end{thm}

\begin{proof} By the soul theorem \cite{CG72}, there is a closed totally geodesic submanifold $S$ of $M$ such that the total space of its normal bundle $\nu$ in $M$ is diffeomorphic to $M$. Since $S$ is totally geodesic, its induced metric also has non-negative sectional curvature. 

Denote by $\hat{S} \xrightarrow{\pi} S$ the universal cover of $S$. Then $\pi^*\nu \to \nu$ is the universal cover of (the total space of) $\nu$. If the universal cover of $M$ is formal, i.e. if $\pi^*\nu$ is formal, then so is $\hat{S}$, as the inclusion of $\hat{S}$ into $\pi^*\nu$ is a homotopy equivalence. Since $S$ is closed and has non-negative Ricci curvature, we conclude by \Cref{universal} that $S$ is formal. Therefore, the total space of $\nu$ is formal, and so $M$ is formal as well.\end{proof}

For example, using the Killing--Hopf theorem, we thus see that every non-negative curvature space form is formal. We immediately have the following analogue to \Cref{largeb1closed}:

\begin{cor}\label{open} Let $M$ be a connected open $n$--manifold with $b_1 \geq n-7$. If $M$ admits a complete Riemannian metric with non-negative sectional curvature, then it is formal. \end{cor}

The bound $b_1 \geq n-7$ comes from the fact that any soul of $M$ has dimension at most $n-1$.

\begin{rem} Throughout the above, instead of invoking the result of formality being preserved under domination \cite{MSZ23}, one could also wish to invoke a statement of the form ``if $X$ is a formal space, acted on by a finite group $G$, then $X/G$ is formal as well''. In the case of a Galois covering $Y' \to Y$, that would then give that formality of $Y'$ implies formality of $Y$. The Cheeger--Gromoll splitting theorem tells us there is a finite cover of a closed non-negatively curved manifold $M$ which is diffeomorphic to $T^k \times N$, where $N$ is simply connected. Now, this finite cover of $M$ may not be Galois, but some further finite cover is, since every finite index subgroup of a group $G$ contains a finite index subgroup which is normal in $G$. A finite covering of $T^k \times N$ is again of the form $T^k \times N$. Hence we could, without loss of generality, replace any finite cover by a Galois finite cover in our arguments.

However, as far as the author can tell, the only known argument (not using \cite{MSZ23}) that $X$ formal implies $X/G$ formal relies on the result that for a formal space $X$ with a $G$-action for a finite group $G$, there is a zigzag of quasi-isomorphisms between $A_{PL}(X)$ and $H^*(X;\mathbb{Q})$, witnessing the formality of $X$, that is simultaneously $G$-equivariant. This is proven for finite-type spaces $X$ with $b_1 = 0$ in \cite[Corollary 2.10, Corollary 2.11]{Pa82}. In \cite[Remark 3.30]{FOT08} it is stated that the result follows from the results of \cite{Op84}. The proofs of \cite{Op84} deal with so-called \emph{nice} spaces, which have (globally) finitely generated minimal models. In both cases, it is stated that the results of the respective papers can be extended (\cite[comment above Corollary 2.10]{Pa82}, \cite[Remark on p.180]{Op84}) to the generality required here, with proofs omitted. For the sake of completeness and as an illustration of the main result therein, we instead employ \cite{MSZ23}. \end{rem}

\subsection*{Acknowledgements} The author thanks Luc\'ia Mart\'in-Merch\'an and Michael Albanese for very helpful discussions, together with Jonas Stelzig, Martin Markl, and Demetre Kazaras for important references. Furthermore, the author is grateful to Igor Belegradek for suggesting the statement of \Cref{universal} and \Cref{universalsectional} upon reading an earlier version of this note, and to the anonymous referees for helpful comments which improved the exposition and simplified the proof of \Cref{universalsectional}.

 

 

\end{document}